\documentclass{article}
\usepackage[utf8]{inputenc}

\usepackage[english]{babel}

\usepackage[utf8]{inputenc}
\usepackage[T1]{fontenc}

\usepackage{amssymb}
\usepackage{stmaryrd}
\usepackage{amsmath,amsfonts}
\usepackage[tikz]{bclogo}
\usepackage[top=4.34cm]{geometry}

\usepackage{lmodern}
\usepackage{textcomp}
\usepackage{mathtools, bm}
\usepackage{amssymb, bm}
\usepackage[unq]{unique}
\usepackage{enumerate}
\usepackage{comment}
\usepackage[breaklinks]{hyperref}

\usepackage[numbers]{natbib}  

\usepackage[breaklinks]{hyperref}
\usepackage[numbers]{natbib}

\usepackage{amsthm}

\usepackage{cleveref}
\usepackage[utf8]{inputenc}
\usepackage[T1]{fontenc}
\usepackage{enumitem}
\usepackage{tipa}
\usepackage{float}
\usepackage{dirtytalk}
\usepackage{pbox}
\usepackage{xcolor}
\usepackage{setspace}
\usepackage[normalem]{ulem}
\usepackage{geometry}
\usepackage{amsmath}
\usepackage{amsthm}
\usepackage{tikz}
\usepackage{amssymb}
\usepackage{multicol}
\usepackage{microtype}
\usepackage{datetime}
\usepackage{endnotes}
\usepackage{babel}
\usepackage[mathscr]{euscript}
\usepackage{bigfoot}
\usetikzlibrary {arrows.meta}
\usetikzlibrary{decorations.pathreplacing,calligraphy,quotes,calc,trees,positioning,arrows,chains,shapes.geometric,
    decorations.pathreplacing,decorations.pathmorphing,shapes,
    matrix,shapes.symbols}

\DeclareMathOperator{\Hroot}{root}
\DeclareMathOperator{\Var}{Var}

\newcommand{\personal}[1]{}
\newtheorem{theorem}{Theorem}[section]
\newtheorem*{theorem*}{Theorem}

\newtheorem{lemma}[theorem]{Lemma}

\newtheorem{corollary}[theorem]{Corollary}

\newtheorem{question}[theorem]{Question}

\newtheorem{observation}[theorem]{Observation}
\newcommand{\brssixthaug}{\color{black}}

\newcommand{\points}[0]{V}
\newcommand{\newdelta}[0]{\varepsilon}
\newcommand{\newT}[0]{Q}
\newcommand{\newr}[0]{\alpha}
\newcommand{\newI}[0]{J}
\newcommand{\whp}[0]{w.h.p.}
\newcommand{\dist}[1]{\mathrm{dist}(#1)}

\title{Reconstructing almost all of a point set in \(\mathbb{R}^d\) from randomly revealed pairwise distances}

\author{Douglas Barnes\thanks{\{db875, jp895, jp899, bwr26, ss2765\}@cam.ac.uk, Department of Pure Mathematics and Mathematical Statistics (DPMMS), University of Cambridge, Wilberforce Road, Cambridge, CB3 0WA, United Kingdom} , 
Jan Petr$^{*}$, 
Julien Portier$^{*}$,
Benedict Randall Shaw$^{*}$,
Alan Sergeev$^{*}$}
\date{}
\begin{document}

\maketitle

	\begin{abstract}
	Let \(\points\) be a set of \(n\) points in \(\mathbb{R}^d\), and suppose that the distance between each pair of points is revealed independently with probability \(p\). We study when this information is sufficient to reconstruct large subsets of \(\points\), up to isometry.
 
    Strong results for \(d=1\) have been obtained by Gir\~ao, Illingworth, Michel, Powierski, and Scott. In this paper, we investigate higher dimensions, and show that if \(p>n^{-2/(d+4)}\), then we can reconstruct almost all of \(\points\) up to isometry, with high probability. We do this by relating it to a polluted graph bootstrap percolation result, for which we adapt the methods of Balogh, Bollob\'as, and Morris.
	\end{abstract}
	\section{Introduction}
	
	Let \(\points\) be a set of \(n\) points in \(\mathbb{R}^d\), and suppose that all we know about $\points$ are the distances between each pair in \(\points\) independently with probability \(p\). That is, the distance between points \(x\) and \(y\) is revealed to us whenever \(\{x,y\}\in\mathcal{P}\subset{\points}^{(2)}\), for \((\points,\mathcal{P})\) distributed as \(G(n,p)\), the Erd\H{o}s-Rényi random graph. We want to determine the range of $p$ for which we can reconstruct \(\points\), or alternatively some large \(Y \subseteq \points\), where \textit{reconstruction} of a set here means determining the positions of its points up to isometry. A set $Y \subseteq \points$ whose position can be determined up to isometry given $(\points,\mathcal{P})$ will be called \emph{reconstructible}.

 Recently, Benjamini and Tzalik \cite{benjamini_determining_2022} proved the following result.

    \begin{theorem}
    There exists an absolute constant \(C\) such that if \(\points\) is a set of \(n\) points in \(\mathbb{R}\) and each pairwise distance of points in $\points$ is revealed independently with probability \(p \geq \frac{C\ln n}{n}\), then $\points$ is reconstructible {\whp}
    \end{theorem}
	
	Gir\~ao, Illingworth, Michel, Powierski, and Scott \cite{girao_reconstructing_2023} improved upon this result, proving a sharp ``threshold'' for reconstructing the whole of $\points$.

	\begin{theorem}
		Let \(\points\) be a set of \(n\) points in \(\mathbb{R}\). Suppose the graph \(G=(\points,\mathcal{P})\) of revealed pairwise distances is distributed as \(G(n,p)\). Then the following hold {\whp}
		\begin{enumerate}[label=(\alph*)]
			\item If $p \leq \frac{\log n+\log \log n -\omega(1)}{n}$, then \(\points\) is not reconstructible.
			\item If $p \geq \frac{\log n+\log \log n +\omega(1)}{n}$, then  \(\points\) is reconstructible.
		\end{enumerate}\label{gimps-whole}
	\end{theorem}

Moreover, the authors of \cite{girao_reconstructing_2023} asked whether similar results could be proved for sets lying in \(\mathbb{R}^d\) for \(d\geq 2\). They observed that we cannot expect to reconstruct the whole of \(\points\) for non-trivial \(p\). Indeed, for an embedding where $n-2$ points lie in a hyperplane and the other two points $u, v$ do not, the whole of $\points$ can only be reconstructed if the distance between $u$ and $v$ is revealed, as otherwise each of $u, v$ could be on any of the two sides of the hyperplane, as shown in \Cref{worstcase} for $d=2$.

	\begin{figure}[H]
		\centering
		\begin{tikzpicture}
		\draw (0,0)--(4,2);
		\filldraw[black] (0,0) circle (2pt) node[anchor=north west]{\(x_1\)};
		\filldraw[black] (0.6,0.3) circle (2pt) node[anchor=north west]{\(x_2\)};
		\filldraw[black] (1.2,0.6) circle (2pt) node[anchor=north west]{\(x_3\)};
		\filldraw[black] (1.6,0.8) circle (2pt) node[anchor=west]{\hspace{1em}\(\cdots\)};
		\filldraw[black] (2.0,1.0) circle (2pt);
		\filldraw[black] (2.2,1.1) circle (2pt);
		\filldraw[black] (2.6,1.3) circle (2pt);
		\filldraw[black] (3.0,1.5) circle (2pt);
		\filldraw[black] (3.2,1.6) circle (2pt);
		\filldraw[black] (3.6,1.8) circle (2pt);
		\filldraw[black] (4,2) circle (2pt) node[anchor=north west]{\(x_{n-2}\)};
		
		\filldraw[red] (0.5,2.5) circle (2pt) node[anchor=south east]{\(u\)};
		\filldraw[red] (2.5,2.5) circle (2pt) node[anchor=south west]{\(v\)};
		\draw[red] (3.5,0.5) circle (2pt);
		\draw[red,dashed] (0.5,2.5)--(2.5,2.5);

		\draw (6,0)--(10,2);
		\filldraw[black] (6,0) circle (2pt) node[anchor=north west]{\(x_1\)};
		\filldraw[black] (6.6,0.3) circle (2pt) node[anchor=west]{\hspace{1em}\(\cdots\)};
		\filldraw[black] (7.2,0.6) circle (2pt);
		\filldraw[black] (7.6,0.8) circle (2pt);
		\filldraw[black] (8.0,1.0) circle (2pt);
		\filldraw[black] (8.2,1.1) circle (2pt);
		\filldraw[black] (8.6,1.3) circle (2pt);
		\filldraw[black] (9.0,1.5) circle (2pt);
		\filldraw[black] (9.2,1.6) circle (2pt);
		\filldraw[black] (9.6,1.8) circle (2pt);
		\filldraw[black] (10,2) circle (2pt) node[anchor=north west]{\(x_{n-2}\)};
		
		\filldraw[red] (6.5,2.5) circle (2pt) node[anchor=south east]{\(u\)};
		\draw[red] (8.5,2.5) circle (2pt);
		\filldraw[red] (9.5,0.5) circle (2pt) node[anchor=north west]{\(v\)};
		\draw[red,dashed] (6.5,2.5)--(9.5,0.5);
;		\end{tikzpicture}
		\caption{These two configurations differ only in the distance marked in red.}\label{worstcase}
	\end{figure}
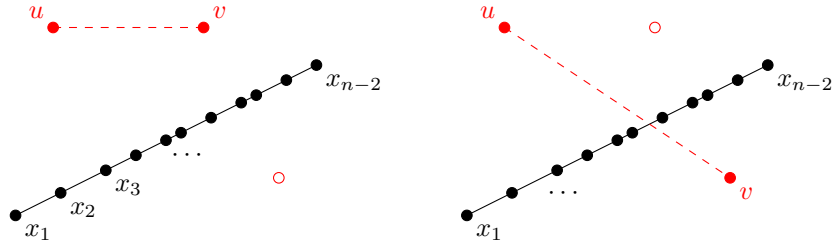

 Therefore, Gir\~ao, Illingworth, Michel, Powierski, and Scott asked the following question.
  \begin{question}
  \label{qu:GIMPS}
  Let $V$ be a set of $n$ points in $\mathbb{R}^d$. For which range of \(p\) does $\points$ contain a reconstructible set of size $\Omega(n)$ {\whp}?
  \end{question}

    In the one-dimensional case $d=1$, they proved the following result.

    \begin{theorem}
    \label{thm:GIMPSLinear}
		Let \(\points\) be a set of \(n\) points in \(\mathbb{R}\). Suppose the graph \(G=(\points,\mathcal{P})\) of revealed pairwise distances is distributed as \(G(n,p)\). 
    \begin{enumerate}[label=(\alph*)]
			\item If $p \geq 42/n$, then $\points$ contains a reconstructible set of size $\Omega(n)$ {\whp}
			\item If \(pn\to\infty\), then $\points$ contains a reconstructible set of size \(n-o(n)\) {\whp}
		\end{enumerate}
	\end{theorem}

 Note that for $p \leq \frac{1}{n}$, a graph sampled as $G(n,p)$ does not have any connected component of linear size {\whp} and so the previous result gives a weak threshold about reconstructibility of a linear size subset.

Our main contribution in this paper is the first non-trivial upper bound on the reconstructibility threshold from \Cref{qu:GIMPS}.
 
    \begin{theorem}
    \label{thm:main}
		Let \(d\geq 1\) and \(\points\) be a set of \(n\) points in \(\mathbb{R}^d\). Suppose the graph \(G=(\points,\mathcal{P})\) of revealed pairwise distances is distributed as \(G(n,p)\). If \(p\gg q(n,d)\), then $\points$ contains a reconstructible set of size \(n-o(n)\) {\whp}, where \[q(n,d):=n^{-1/\eta(d)+o(1)} \qquad \text{and} \qquad
        \eta(d):=\frac{d+4}2-\frac1{d+1}.\]
	\end{theorem}

    Although Benjamini and Tzalik \cite{benjamini_determining_2022} were the first to study those problems in this setting, a long line of research has already been conducted in the case of \emph{generic} point sets. A set $\points \subset \mathbb{R}^d$ is said to be \emph{generic} if the $d|\points|$ coordinates of the vertices are algebraically independent over $\mathbb{Q}$. 
    An embedding $f \colon \points \rightarrow \mathbb{R}^d$ is called \emph{generic} if its image is generic.
    A graph \((\points,\mathcal{P})\) is \emph{globally rigid in $\mathbb{R}^d$} if it has some generic embedding
    in $\mathbb{R}^d$ which is reconstructible from its edge lengths, i.e. the distances \(uv\) such that \(uv\in\mathcal{P}\). 
    It has been shown \cite{connelly2005generic, gortler_characterizing_2012} that a graph is globally rigid if and only if all
    of its generic embeddings are reconstructible from their edge lengths.
    
    The global rigidity of random graphs has attracted significant attention. For instance, Lew, Nevo, Peled and Raz \cite{lew_sharp_2022} proved that \(G(n,p)\) is globally \(d\)-rigid {\whp}\ exactly when it has minimum degree \(d+1\), which has a sharp threshold at \(p = \frac{\log{n}+ d\log{\log{n}}}{n}\). 
    We would like to remark that the restriction to generic embeddings in the definition of global rigidity is a significant weakening. For example, it is folklore (see for example Theorem 63.2.7 in \cite{jordan2017global}) that a graph is globally rigid in $\mathbb{R}$ if and only if it is $2$-connected, which is not the case for reconstructing arbitrary point sets. Indeed, Gir\~ao, Illingworth, Michel, Powierski, and Scott showed \cite{girao_reconstructing_2023} that there are graphs with arbitrarily high connectivity which can be embedded in $\mathbb{R}$ in such a way that their vertex sets cannot be reconstructed from their edge lengths.

 Finally, we note that the scope of reconstruction problems is quite broad. For example, Lemke, Skiena and Smith \cite{lemke_reconstructing_2003} considered the problem of reconstructing a point set based on pairwise distances, where the correspondence between distances and the pairs of points to which they belong is not known. As mentioned in their paper, this problem has natural applications in the worlds of X-ray crystallography and site mapping of DNA. On the other hand, the problems of reconstructing a point set based on labelled distances have applications in network localisation \cite{eren_rigidity_2004} and molecular conformation \cite{hendrickson_molecule_1990}.

    \subsection{Proof outline}

    In this paper, we view \(\mathbb{R}^m\) as an affine space, in that the origin does not play a special role. A \emph{subspace} will mean an affine subspace. Given $d$, we call a multiset $X$ of $d+1$ points \emph{$d$-dependent} if there is a subspace of dimension $d-1$ that contains $X$. If there is no such subspace, we will say that $X$ is \emph{$d$-independent}. Any multiset of size $1$ is $0$-independent.

    The main idea of our proof is based on the following observation, proved in a stronger form as \Cref{cor:GeometryFact}.

    \begin{observation}
    \label{obs:Percolation}
    Let \(u,v, x_1, x_2, ... x_{d+1} \in \mathbb{R}^d\) be points for which \(\{x_i|i\in [d+1]\}\) is \(d\)-independent. If all pairwise distances known except \(uv\) are known, we can reconstruct the distance \(uv\).
    \end{observation} 
    Therefore if we knew that $\points$ was in general position, i.e. with no set of \(d+1\) points being \(d\)-dependent, then repeatedly applying \Cref{obs:Percolation} would allow us to follow a \mbox{\(K_{d+3}\)-percolation}, defined in Section \ref{sec:tools}, on the graph of known edges. We would then immediately deduce \Cref{thm:main} by \Cref{thm:Balogh}, a general result about graph bootstrap percolation by Balogh, Bollob\'as, and Morris.
    For a set of points $\points$ not necessarily in general position, we still want to apply a \(K_{d+3}\)-bootstrap percolation on the graph of known distances, but only reconstructing \(uv\) from copies of \(K_{d+3}\) such that \(V(K_{d+3})\setminus\{u,v\}\) is \(d\)-independent, so that \Cref{obs:Percolation} applies. 
    

    To account for \(d\)-dependent sets, we will work with \emph{polluted graph bootstrap percolation}, defined in \Cref{sec:PollutedBootstrap}, which is similar to the classical graph bootstrap percolation, except that there is a set of hyperedges $E(\mathcal{H})$ on which our bootstrap percolation cannot spread.
    In our reconstruction problem, we will take $E(\mathcal{H}) \subset V^{(d+1)}$ to be exactly the $d$-dependent subsets of $V$ of size $d+1$. 
    We then show that \Cref{thm:Balogh} essentially still holds for the same range of $p$ for a polluted graph bootstrap percolation if the set of hyperedges $E(\mathcal{H})$ is small enough, see \Cref{thm:PollutedBootstrapPercolation}.
    Then we can deduce our main result, \Cref{thm:main}, from this polluted bootstrap percolation result, \Cref{thm:PollutedBootstrapPercolation}, provided $|E(\mathcal{H})|$ is small enough in terms of $n$. In the general case, $|E(\mathcal{H})|$ is not always sufficiently small to apply \Cref{thm:PollutedBootstrapPercolation}, and consequently we use \Cref{cor:Dichotomy2} to find a reasonably large subset of $\points$ where we can apply \Cref{thm:PollutedBootstrapPercolation}.
    
    More precisely, we show in \Cref{cor:Dichotomy2} that for every multiset of points $\points$, there exists a subspace $\Pi$ of dimension \(d'\) which contains reasonably many points of $\points$, and such that most families of $\Pi \cap \points$ of size $d'+1$ are \(d'\)-independent. The particular case where $d'=0$ needs to be taken care of separately, but in the general case where $d' \geq 1$, we can indeed reconstruct the pairwise distances between almost all of $\Pi \cap \points$ by using our polluted graph bootstrap percolation result in \Cref{thm:PollutedBootstrapPercolation}. We then show that we can reconstruct the projections of almost all of $\points$ onto the subspace $\Pi$ in \Cref{lem:RecoveringDistanceToHyperplane}.
    Combining all of the above results, we obtain \Cref{lem:InductionTechnicalLemma}, which enables us to reduce the problem to one with \(V\) projected onto \(\Pi^\perp\) as its point set, reducing the dimension of the problem by at least $1$. By sprinkling and iterating this at most $d$ times, we are able to deduce \Cref{thm:main}.
    
    The rest of the paper is structured as follows.
    In \Cref{sec:tools} we introduce some tools. 
    In \Cref{sec:PollutedBootstrap} we prove our polluted graph bootstrap percolation result. In \Cref{sec:ProofMainResult} we proceed to the proof of our main result.
    Finally, in \Cref{sec:Conclusion}, we give a few concluding remarks.

    \brssixthaug

    \section{Tools from geometry and graph theory}
    \label{sec:tools}
Graph bootstrap percolation was introduced over 50 years ago by Bollob\'as \cite{bollobas1968weakly} as a monotone case of cellular automata, introduced a few years earlier by von Neumann \cite{neumann1966theory} after a suggestion of Ulam \cite{ulam1952random}. For a given graph \(H\), the \(H\)-\emph{bootstrap percolation process} is defined as follows. Given an initial set of edges \(G\subset E\left(K_n\right)\), and graph \(H\), we set \(G_0=G\), and thereafter define
\[G_{t+1}=G_t\cup\left\{e\in E\left(K_n\right):\exists K \cong H \mbox{ s.t. } e \in E(K)\subset G_t \cup \{e\}\right\}.\]
 We then define \(\langle G\rangle_H=\cup_tG_t\) to be the closure of \(G\) under this process. This is to say, \(\langle G\rangle_H\) is the graph obtained by adding an edge wherever doing so would create a new copy of \(H\), and repeating this on the resulting graph until no new edges are added. A graph \(G\) is said to $H$-\emph{percolate} if \(\langle G\rangle_H=E(K_n)\).

    \begin{figure}[H]
		\centering
		\begin{tikzpicture}[scale=0.65]
		\filldraw[black] (-6,1.1547) circle (2pt);
		\filldraw[black] (-6,-1.1547) circle (2pt);
		\filldraw[black] (-4,2.3094) circle (2pt);
		\filldraw[black] (-4,-2.3094) circle (2pt);
		\filldraw[black] (-2,1.1547) circle (2pt);
		\filldraw[black] (-2,-1.1547) circle (2pt);
		\filldraw[black] (0,1.1547) circle (2pt);
		\filldraw[black] (0,-1.1547) circle (2pt);
		\filldraw[black] (2,2.3094) circle (2pt);
		\filldraw[black] (2,-2.3094) circle (2pt);
		\filldraw[black] (4,1.1547) circle (2pt);
		\filldraw[black] (4,-1.1547) circle (2pt);
		\filldraw[black] (6,1.1547) circle (2pt);
		\filldraw[black] (6,-1.1547) circle (2pt);
		\filldraw[black] (8,2.3094) circle (2pt);
		\filldraw[black] (8,-2.3094) circle (2pt);
		\filldraw[black] (10,1.1547) circle (2pt);
		\filldraw[black] (10,-1.1547) circle (2pt);
		\filldraw[black] (12,1.1547) circle (2pt);
		\filldraw[black] (12,-1.1547) circle (2pt);
		\filldraw[black] (14,2.3094) circle (2pt);
		\filldraw[black] (14,-2.3094) circle (2pt);
		\filldraw[black] (16,1.1547) circle (2pt);
		\filldraw[black] (16,-1.1547) circle (2pt);
        \node at (-4,-3) {\(G=G_0\)};
        \draw [->](-1.5,0) -- (-0.5,0);
        \node at (2,-3) {\(G_1\)};
        \draw [->](4.5,0) -- (5.5,0);
        \node at (8,-3) {\(G_2\)};
        \draw [->](10.5,0) -- (11.5,0);
        \node at (14,-3) {\(G_3=\langle G\rangle_{K_4}=K_6\)};
        
		\draw (-4,2.3094) -- (-6,1.1547) -- (-6,-1.1547) -- (-4,-2.3094) -- (-2,-1.1547) -- (-2,1.1547) -- (-4,2.3094) -- (-6,-1.1547) -- (-2,-1.1547);
        \draw (-6,1.1547) -- (-4,-2.3094);
  
		\draw (2,2.3094) -- (0,1.1547) -- (0,-1.1547) -- (2,-2.3094) -- (4,-1.1547) -- (4,1.1547) -- (2,2.3094) -- (0,-1.1547) -- (4,-1.1547);
        \draw (0,1.1547) -- (2,-2.3094);
        \draw[red] (2,2.3094) -- (2,-2.3094);
        \draw[red] (0,1.1547) -- (4,-1.1547);

        \draw (6,1.1547) -- (8,-2.3094) -- (8,2.3094) -- (6,1.1547) -- (6,-1.1547) -- (8,-2.3094) -- (10,-1.1547) -- (10,1.1547) -- (8,2.3094) -- (6,-1.1547) -- (10,-1.1547) -- (6,1.1547);
        \draw[red] (8,2.3094) -- (10,-1.1547);

        \draw (16,-1.1547) -- (12,1.1547) -- (14,-2.3094) -- (14,2.3094) -- (12,1.1547) -- (12,-1.1547) -- (14,-2.3094) -- (16,-1.1547) -- (16,1.1547) -- (14,2.3094) -- (12,-1.1547) -- (16,-1.1547) -- (14,2.3094);
        \draw[red] (14,-2.3094) -- (16,1.1547) -- (12,1.1547);
        \draw[red] (16,1.1547) -- (12,-1.1547);

				\end{tikzpicture}
		\caption{The process of \(K_4\)-bootstrap percolation for the graph \(G_0\), which \(K_4\)-percolates.}
    \end{figure}
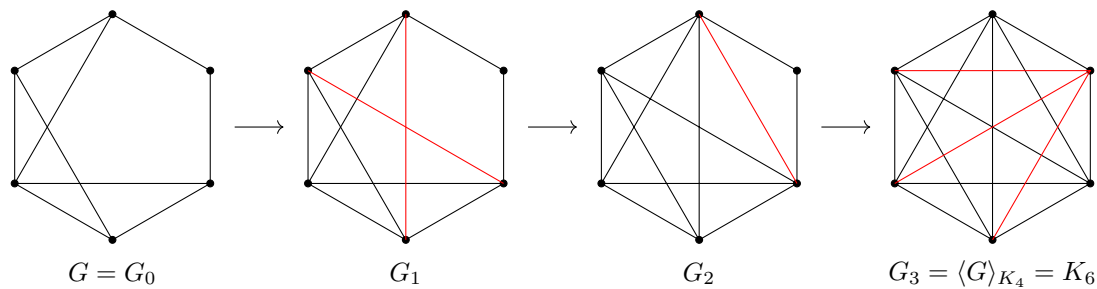
    
    Balogh, Bollob\'as, and Morris \cite{balogh_graph_2012} studied the critical threshold for \(G(n,p)\) to \(H\)-percolate. More precisely, we define the \textit{critical threshold} to be
	\[p_c(n,H):=\inf\left\{p:\mathbb{P}\left(G\sim G(n,p)\text{ }H\text{-percolates}\right)\geq\frac12\right\}.\]
 
	Balogh, Bollob\'as, and Morris determined the critical threshold for $K_{d+3}$-percolation up to a polylogarithmic factor \cite{balogh_graph_2012}.
	
	\begin{theorem}
 \label{thm:Balogh}
		For \(d\geq 1\) and sufficiently large \(n\), there exists \(c = c(d) > 0\) such that \begin{equation}
		\frac{n^{-1/\eta(d)}}{c\log{n}} \leq p_c\left(n,K_{d+3}\right)\leq n^{-1/\eta(d)}\log n. \end{equation} 
	\end{theorem}

	They further remarked \cite{balogh_graph_2012} that by a general result of Bollob\'as and Thomason \cite{bollobas_threshold_1987} this is a threshold, i.e. for \(p\gg p_c(n,H)\), the probability of percolation is \(1-o(1)\), and for \(p\ll p_c(n,H)\), the probability of percolation is \(o(1)\). We will not use \Cref{thm:Balogh} directly, but we will use the proof method of Balogh, Bollob\'as, and Morris
 to deduce our polluted graph bootstrap percolation result in \Cref{thm:PollutedBootstrapPercolation}.
	
	\brssixthaug

    We will need the two following geometric results. Their proofs are quite classical, and therefore deferred to \Cref{appendix}.
	
	\begin{lemma}
    \label{lem:GeometryFact}
        Given all pairwise distances within some \(S=\{v_1,\dots,v_r\}\subset\mathbb{R}^d\), we can reconstruct the configuration of \(S\) up to isometry.
	\end{lemma}
	
	\begin{corollary}
    \label{cor:GeometryFact}
		Suppose a multiset \(R\) of \(d+3\) points of \(\mathbb{R}^m\) \((m\geq d\)) lies in a \(d\)-dimensional subspace. Then, given all pairwise distances within \(R\) except some pair \(uv\), it is possible to identify whether \(\newT=R\setminus\{u,v\}\) is \(d\)-independent. If $Q$ is $d$-independent, the distance \(uv\) can also be determined.
	\end{corollary}

    We will also employ a classical ``supersaturation'' result due to Erd\H{o}s and Simonovits \cite{supersaturation}. Recall that Tur\'an's theorem says that if $G$ is a graph with more than $\frac{\ell -2}{2(\ell -1)}n^2$ edges then it contains a copy of $K_{\ell}$. This theorem says that we in fact have many copies of $K_{\ell}$ if the density is slightly higher.
	
	\begin{theorem}\label{supersat}
		\label{thm:CliqueDensity}
		If $\ell \geq 3$ and $\gamma > \frac{\ell-2}{2(\ell-1)}$, then every graph on $n$ vertices and at least $\gamma n^2$ edges contains $\Omega_{\gamma}(n^\ell)$ many $\ell$-cliques.
	\end{theorem}

 Additionally, we will make use of one more classical result, on the connectivity of random graphs due to Erd\H{o}s and R{\'e}nyi \cite{erdHos1960evolution} (see also Theorem 7.3 of the book by Bollob\'as \cite{bollobas1998random} for a sharp result).

\begin{theorem}
\label{thm:ConnectivityErdosRenyi}
    Let $\varepsilon>0$ be fixed and $p \geq (1+\varepsilon)\frac{\log n}{n}$. Then {\whp}\ the random graph $G(n,p)$ is connected.
\end{theorem}

    \brssixthaug
	
    \section{Polluted graph bootstrap percolation}
    \label{sec:PollutedBootstrap}

    We define $K_{d+3}$-\textit{percolation with pollution} \(\mathscr{P}\subset X^{(d+1)}\), or \(\mathscr{P}\)-\textit{polluted} \(K_{d+3}\)-\textit{percolation} as follows: starting with a graph \(G\), we repeatedly add to the graph any edge \(e=uv\) where adding $e$ would create another copy of \(K_{d+3}\) in the graph, on vertices \(\newT\cup \{u,v \}\) for some \(\newT\in X^{(d+1)}\setminus\mathscr{P}\). More formally, given an initial set of edges \(G\subset X^{(2)}\), we set \(G_0=G\), and thereafter define
 \[G_{t+1}=G_t\cup\left\{e=uv\in X^{(2)}:\exists K \cong K_{d+3} \mbox{ s.t. } e \in E(K)\subset G_t \cup \{e\} \mbox{ and } V(K)\setminus \{u,v\} \notin \mathscr{P}\right\}.\]
	Then \(\langle G\rangle^{\mathscr{P}}_{K_{d+3}}=\cup_tG_t\) is the closure of \(G\) under this process.\footnote{Note that our definition of polluted graph bootstrap percolation is unrelated to the one of bootstrap
percolation in a polluted environment, which was introduced by Gravner and McDonald \cite{gravner1997bootstrap}.}

Throughout the paper, we let $p_{*}(n,d) = (\frac{\log n}{\log \log n})^{2/\eta(d)}n^{-1/\eta(d)}$. The main result of this section is the following.

\begin{theorem}
\label{thm:PollutedBootstrapPercolation}
    Let $0<\mu <1$, $\delta>0$, and let $\mathcal{H}$ be a \((d+1)\)-uniform hypergraph on $[n]$ such that $|E(\mathcal{H})| \leq n^{d+\mu}$.
    Let $G \sim G(n,p)$. If $p \gg p_{*}(n,d)$, then $\langle G \rangle^{\mathcal{H}}_{K_{d+3}}$ contains a clique of size $(1-\delta)n$ {\whp}
\end{theorem}

Before proving \Cref{thm:PollutedBootstrapPercolation}, we introduce a little more notation.
Following \cite{balogh_graph_2012} for $H=K_{d+3}$, define $H_r$ as follows:
Let \(H^{(1)},\dots,H^{(r)}\) be copies of \(H\), and choose edges \(e_1=u_1v_1,\dots,e_r=u_rv_r\) of \(H\) such that for each \(i\), the edges \(e_i,e_{i+1}\) share no endpoints. Then for each \(i=2,\dots,r\), remove the edge \(e_i\) from \(H^{(i-1)}, H^{(i)}\), and identify its endpoints in \(H^{(i-1)}\) with those in \(H^{(i)}\). Finally, remove \(e_1\) from \(H^{(1)}\). The edge \(e_1\) is said to be the \textit{root} of \(H_r\), written \(\Hroot(H_r)=e_1\). Additionally, for each $i\leq r$, we call $V(H^{(i)}) \setminus \{u_i,v_i\}$ the \emph{base} of $H^{(i)}$.
	
	\begin{figure}[H]
		\centering
		\begin{tikzpicture}[decoration={calligraphic brace, amplitude=6pt,raise=2pt,mirror}]
		\filldraw[black] (0,1.1547) circle (2pt);
		\filldraw[black] (0,-1.1547) circle (2pt);
		\filldraw[black] (2,2.3094) circle (2pt);
		\filldraw[black] (2,-2.3094) circle (2pt);
		\filldraw[black] (4,1.1547) circle (2pt);
		\filldraw[black] (4,-1.1547) circle (2pt);
		\filldraw[black] (6,2.3094) circle (2pt);
		\filldraw[black] (6,-2.3094) circle (2pt);
		\filldraw[black] (8,1.1547) circle (2pt);
		\filldraw[black] (8,-1.1547) circle (2pt);
		\filldraw[black] (10,2.3094) circle (2pt);
		\filldraw[black] (10,-2.3094) circle (2pt);
		\filldraw[black] (12,1.1547) circle (2pt);
		\filldraw[black] (12,-1.1547) circle (2pt);
		\draw (0,1.1547) -- (12,1.1547)  -- (6,-2.3094) -- (0,1.1547);
		\draw (0,-1.1547) -- (12,-1.1547)  -- (6,2.3094) -- (0,-1.1547);
		\draw (2,2.3094) -- (2,-2.3094) -- (10,2.3094) -- (10,-2.3094) -- (2,2.3094);
		\draw (0,1.1547) -- (2,2.3094) -- (0,-1.1547) -- (2,-2.3094) -- (0,1.1547);
		\draw (2,2.3094) -- (4,-1.1547) -- (6,2.3094) -- (8,-1.1547) -- (10,2.3094) -- (12,-1.1547) -- (12,1.1547) -- (10,-2.3094) -- (8,1.1547) -- (6,-2.3094) -- (4,1.1547) -- (2,-2.3094);
		\draw (6,2.3094) -- (6,-2.3094);
		
		\draw (10,2.3094) -- (12,1.1547);
		\draw (10,-2.3094) -- (12,-1.1547);

		\draw[red, dashed] (0,-1.1547) -- node[left] {\(e_1\)} ++(0,2.3094);
		\draw[red, dashed] (4,-1.1547) -- node[left] {\(e_2\)} ++(0,2.3094);
		\draw[red, dashed] (8,-1.1547) -- node[left] {\(e_3\)} ++(0,2.3094);	

        \draw [decorate,very thick]
            (4,2.5) -- (0,2.5) node [black,midway,yshift=0.5cm] {$H^{(1)}$};
        \draw [decorate,very thick]    
            (8,2.5) -- (4,2.5) node [black,midway,yshift=0.5cm] {$H^{(2)}$};
        \draw [decorate,very thick]
            (12,2.5) -- (8,2.5) node [black,midway,yshift=0.5cm] {$H^{(3)}$};
		
				\end{tikzpicture}
		\caption{\(H_3\), for \(H=K_6\). Removed edges are shown in dashed red.}
	\end{figure}
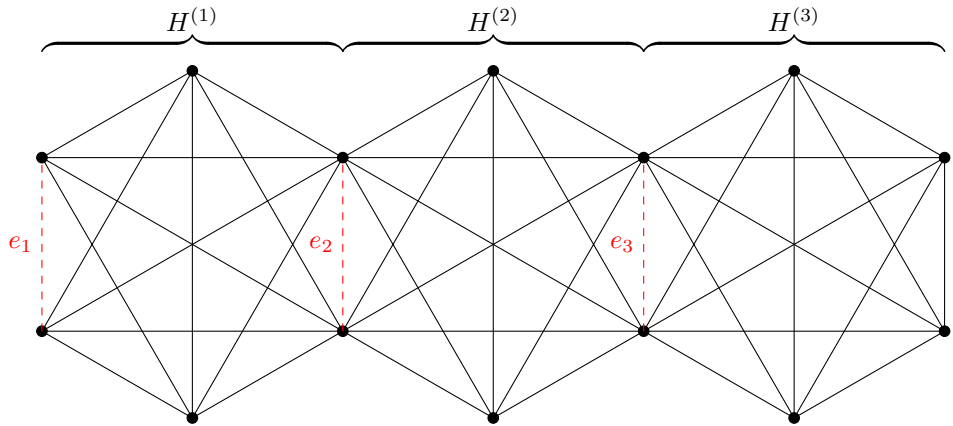

	We say \(H_r\) is $\mathcal{H}$-\textit{clean} if for each \(1\leq i\leq r\), \(V\left(H^{(i)}\right)\setminus \{u_i,v_i\}\notin \mathcal{H}\). Notice that in this case, by \Cref{cor:GeometryFact} we can percolate with pollution \(\mathcal{H}\) on \(H_r\) to obtain edges \(e_r,\dots,e_1\) in that order. Now we define \(X_r(e)\) to be the number of embeddings of \(H_r\) in \(G\sim G(n,p)\) rooted at \(e\), and \(X^{\mathcal{H}}_r(e)\) to be the number of $\mathcal{H}$-clean embeddings of \(H_r\) in \(G\) rooted at \(e\). Notice that if \(X^{\mathcal{H}}_r(e)>0\), then \(e\in\langle G\rangle^{\mathcal{H}}_{K_{d+3}}\). Note that we are counting different embeddings of the graph \(H_r\) as distinct, e.g., a $K_3$ contains six embeddings of a $K_3$.
 
    Balogh, Bollob\'as, and Morris \cite{balogh_graph_2012} proved the following key lemma towards their goal of proving \Cref{thm:Balogh}.
	
	\begin{lemma} 
 Let \(p=p(n), r=r(n), \omega=\omega(n)\) be chosen such that
		\begin{enumerate}[label=(\alph*)]
			\item \(p^{\eta(d)}n\geq(d+3)\omega r,\)
			\item \(\omega^{(d+1)r}\geq n,\)
			\item \(pn\to\infty\) as \(n\to\infty,\) and
			\item \(v(H_r)^{-2}p^{\eta(d)}n\to\infty\) as \(n\to\infty\).
		\end{enumerate}
	Then for $e \in [n]^{(2)}$ and $G \sim G(n,p)$, we have \[\mathbb{E}\left(X_r(e)\right)\to\infty \qquad \text{and} \qquad \frac{\Var\left(X_r(e)\right)}{\mathbb{E}\left(X_r(e)\right)^2} <\left(\frac{p^{\eta(d)}n}{4v(H_r)^2}-1\right)^{-1}\to0\]
 as \(n\to\infty\).\label{bbmconditions}
	\end{lemma}

	We prove that under additional conditions, this result also holds for the polluted version of \(X^{\mathcal{H}}_r(e)\), when $|E(\mathcal{H})|$ is not too large.
	
	\begin{lemma}
 \label{lem:PollutedBootstrapPercolation}
		Let $0<\mu <1$ be fixed such that $|E(\mathcal{H})| \leq n^{d+\mu}$ and let \(p=p(n), r=r(n), \omega=\omega(n)\) meet the conditions of Lemma \ref{bbmconditions}, as well as
		
		\begin{enumerate}[label=(\alph*)]
			\setcounter{enumi}{4}
			\item \(r(d+1)\leq\frac12n\), and
			\item  \(rn^{\mu-1}\leq c\) for some constant \(c\).
		\end{enumerate}
		Then there exists a constant $L_{d,c}$ depending only on $d$ and $c$ such that for $e \in [n]^{(2)}$ and $G \sim G(n,p)$, we have
		
		\[\mathbb{E}\left(X^{\mathcal{H}}_r(e)\right) \to\infty\qquad \text{and} \qquad \dfrac{\Var\left(X^{\mathcal{H}}_r(e)\right)}{\mathbb{E}\left(X^{\mathcal{H}}_r(e)\right)^2}< L_{d,c}\left(\frac{p^{\eta(d)}n}{4v(H_r)^2}-1\right)^{-1}\to0.\]
  as \(n\to\infty\).
	\end{lemma}

	\begin{proof}
 
		Define \(T=\left\{E\left(K\right):K\text{ is an embedding of }H_r\text{ in } K_n \text{ rooted at }e\text{ with vertices in }X\right\}\), and \(T^{\mathcal{H}}\subset T\) to be \(T^{\mathcal{H}}=\left\{E\left(K\right):K\text{ is a $\mathcal{H}$-clean embedding of }H_r\text{ rooted at }e\text{ with vertices in }X\right\}\). Now we can write
		\(X_r(e)=\sum_{S\in T}\mathbf{1}\left(S\subset E\left(G\right)\right)\), and likewise \(X^{\mathcal{H}}_r(e)=\sum_{S\in T^{\mathcal{H}}}\mathbf{1}\left(S\subset E\left(G\right)\right)\), from which we can derive
		\[\begin{split}\Var\left(X_r(e)\right)&{}=\sum_{S_1,S_2\in T}\Bigl(\mathbb{P}\bigl(S_1\cup S_2\subset E\left(G\bigr)\right)-\mathbb{P}\bigl(S_1\subset E\left(G\right)\bigr)\mathbb{P}\bigl(S_2\subset E\left(G\right)\bigr)\Bigr)\\&{}=\sum_{S_1,S_2\in T}\left(p^{\left|S_1\cup S_2\right|}-p^{2e\left(H_r\right)}\right)\\&{}\geq\sum_{S_1,S_2\in T^{\mathcal{H}}}\left(p^{\left|S_1\cup S_2\right|}-p^{2e\left(H_r\right)}\right)=\Var\left(X^{\mathcal{H}}_r(e)\right).\end{split}\]
		
		But now \(\mathbb{E}\left(X_r(e)\right)=\left|T\right|p^{e\left(H_r\right)}\), and likewise \(\mathbb{E}\left(X^{\mathcal{H}}_r(e)\right)=\left|T^{\mathcal{H}}\right|p^{e\left(H_r\right)}\). We choose an element of \(T\) by choosing in turn \(V(H^{(r)})\setminus e_r ,V(H^{(r-1)})\setminus e_{r-1},\dots,V(H^{(2)})\setminus e_2, V(H^{(1)})\setminus e_1\). Note that these are the bases of the \(H^{(i)}\), and that \(e_1=e\) is already fixed. Thus

        \[\left|T\right|=\left(\prod_{i=0}^{r-1}{{n-i(d+1)}\choose{d+1}} \right)\left((d+1)!\right)^r .\]

  By a similar counting, making sure that the base of $H^{(i)}$ is never an element of $E(\mathcal{H})$, we have
  
        \[\left|T^{\mathcal{H}}\right| \geq \left(\prod_{i=0}^{r-1} \left({{n-i(d+1)}\choose{d+1}}-|E(\mathcal{H})|\right) \right)\left((d+1)!\right)^r .\]

        Therefore, using $|E(\mathcal{H})| \leq n^{d+\mu}$, the inequality ${{N}\choose{k}} \geq (\frac{N}{k})^k$, and condition (e), we have 
        \begin{align*}
            |T^{\mathcal{H}}| &\geq |T| \prod_{i=0}^{r-1} \left(1-\frac{n^{d+\mu}}{(\frac{n/2}{d+1})^{d+1}} \right) \\
            &= |T| \left(1-\frac{(2(d+1))^{d+1}}{n^{1-\mu}} \right)^r.
        \end{align*}
		
		Note that for small enough \(x>0\), we have \((1-x)\geq e^{-2x}\). Thus for large \(n\), condition (f) gives
		
		\[\frac{\mathbb{E}\left(X^{\mathcal{H}}_r(e)\right)}{\mathbb{E}\left(X_r(e)\right)}=\frac{\left|T^{\mathcal{H}}\right|}{\left|T\right|}\geq\left(1-\frac{(2(d+1))^{d+1}}{n^{1-\mu}} \right)^r\geq e^{-2(2(d+1))^{d+1}rn^{\mu-1}}\geq e^{-2(2(d+1))^{d+1}c}=L_{d,c}^{-1/2},\]

where $L_{d,c}=e^{4(2(d+1))^{d+1}c}$ is a constant.
		Hence \(\mathbb{E}\left(X^{\mathcal{H}}_r(e)\right)\to\infty\), and
		
		\[\frac{\Var\left(X^{\mathcal{H}}_r(e)\right)}{\mathbb{E}\left(X^{\mathcal{H}}_r(e)\right)^2}\leq L_{d,c} \frac{\Var\left(X_r(e)\right)}{\mathbb{E}\left(X_r(e)\right)^2}<L_{d,c} \left(\frac{p^{\eta(d)}n}{4v(H_r)^2}-1\right)^{-1}\to0.\]
	\end{proof}
 
We now proceed to the proof of \Cref{thm:PollutedBootstrapPercolation}.
 \begin{proof}[Proof of \Cref{thm:PollutedBootstrapPercolation}.]
    Let $G'=\langle G \rangle^{\mathcal{H}}_{K_{d+3}}$. We may assume that $\delta^2 (1-2\delta)^{-2} < (2d)^{-1}$.
    We let $p \gg p_{*}(n,d)$, $r=\omega= \left\lceil \frac{\log n}{\log \log n} \right\rceil$ and apply \Cref{lem:PollutedBootstrapPercolation}. The authors of \cite{balogh_graph_2012} have already shown that conditions (a) to (d) are satisfied, and it is clear that the additional conditions (e) and (f) are also satisfied. It follows by Chebyshev's inequality that 
    $$\mathbb{P}\left( e \notin G' \right)=\mathbb{P}\left(X^{\mathcal{H}}_r(e)=0\right)\leq\frac{\Var\left(X^{\mathcal{H}}_r(e)\right)}{\mathbb{E}\left(X^{\mathcal{H}}_r(e)\right)^2}<L_{d,c}\left(\frac{p^{\eta(d)}n}{4v(H_r)^2}-1\right)^{-1}\to0.$$ 

    Hence, by Markov's inequality, $G'$ has at least $n^2/2-\delta^2 n^2/2$ edges {\whp} We now show that {\whp},  $G'$ contains a clique on $(1-\delta)n$ vertices by combining a similar claim and its proof from \cite{balogh_graph_2012} with the classical supersaturation result in \Cref{thm:CliqueDensity}.

    Let \(D=\left\{x\in [n]:d_{G'}(x)>(1-\delta)n\right\}\). Then
    \begin{align*}
        n^2-\delta^2 n^2 &\leq 2|E(G')| \\
        &= \sum_{v \in D} d_{G'}(v) + \sum_{v \in D^c} d_{G'}(v) \\
        &\leq |D|n+(n-|D|)(1-\delta)n \\
        &= |D|\delta n + (1-\delta)n^2.
    \end{align*}
    Rearranging, this gives $|D| \geq (1-\delta)n$.
    
    We now claim that $D$ is a clique for sufficiently large $n$. Suppose for contradiction that there exist \(x,y\in D\) such that there is no edge between $x$ and $y$ in $G'$. Let \(n_{*}=\left|N_{G'}(x)\cap N_{G'}(y)\right| \geq (1-2\delta)n\), and let $G^{*}$ be the induced subgraph of $G'$ on vertex set $N_{G'}(x)\cap N_{G'}(y)$. For sufficiently large $n$ we have 
    \begin{align*}
        |E(G^{*})| &\geq \binom{n_*}{2}-\delta^2 n^2/2 \\
        &\geq \binom{n_*}{2} - \delta^2 (1-2\delta)^{-2} n_{*}^2/2 \\
        &\geq \left (\frac{1}{2}-\delta^2(1-2\delta)^{-2} \right) n_*^2.
    \end{align*}
    Let $\gamma=1/2- \delta^2 (1-2\delta)^{-2}$. By \Cref{thm:CliqueDensity}, the graph $G^{*}$ contains \(\Omega_{\gamma}(n^{d+1})\) many \((d+1)\)-cliques. As $|E(\mathcal{H})| \leq n^{d+\mu}$, there exists a $(d+1)$-clique in $N_{G'}(x)\cap N_{G'}(y) \cap E(\mathcal{H})^c$. This is a contradiction. Hence, $D$ is a clique, which finishes the proof.
\end{proof} 

 \section{Proof of \Cref{thm:main}}
 \label{sec:ProofMainResult}
 We are now equipped with \Cref{thm:PollutedBootstrapPercolation}. 
 We aim to use it with the hyperedges corresponding to $d$-dependent tuples of size $d+1$, utilizing \Cref{obs:Percolation}. 
 We show in the following geometric lemma that either the number of $d$-dependent tuples of size $d+1$ is small enough to apply \Cref{thm:PollutedBootstrapPercolation}, or $\points$ has a large intersection with a subspace of lower dimension. 
 This dichotomy will enable us to then deduce \Cref{cor:Dichotomy2}, which roughly states that there always is a large subset of $\points$ lying in some (possibly non-proper) subspace of $\mathbb{R}^d$ for which we can apply \Cref{thm:PollutedBootstrapPercolation}.

 \begin{lemma}
 \label{lem:DichotomyHyperplaneDependent}
     Let $0 <\mu < 1$ and $\points$ be a multiset of $n$ points in $\mathbb{R}^d$. Then either
     \begin{itemize}
     \item there exists a subspace $\Pi_{d'}$ of dimension $0 \leq d' \leq d-1$ such that $|V \cap \Pi_{d'}| \geq n^{\mu}$, or
     \item the number of $d$-dependent tuples of $\points$ of size $d+1$ is at most $d n^{d+\mu}$.
     \end{itemize}
 \end{lemma}
 
\begin{proof}
Suppose that there is no subspace $\Pi_{d'}$ of dimension $0 \leq d' \leq d-1$ such that $|V \cap \Pi_{d'}| \geq n^{\mu}$. Let $I$ be the set of $d$-dependent tuples of $V$ of size $d+1$, and for each $0 \leq i \leq d-1$, let $I_i$ be the set of families of $V$ of size $d+1$ whose affine span has dimension exactly $i$. We have $I= \bigcup_{i=0}^{d-1} I_i$ and therefore $|I| = \sum_{i=0}^{d-1} |I_i|$. By definition, for each element $S$ of $I_i$, there exist $x_1, \dots, x_i, x_{i+1} \in S$ such that the subspace $\Pi_S$ spanned by $x_1, \dots, x_i, x_{i+1}$ has dimension exactly $i$, and $S\setminus \{ x_1, \dots, x_{i+1} \}$ lies in $\Pi_S$. As we assumed that no subspace $\Pi_S$ of dimension $0 \leq d' \leq d-1$ satisfies $|V \cap \Pi_S| \geq n^{\mu}$, it follows that once $x_1, \dots, x_{i+1}$ are fixed, there are at most $n^{(d-i)\mu}$ possibilities for $S\setminus \{ x_1, \dots, x_{i+1} \}$. Therefore, $|I_i| \leq n^{i+1+(d-i)\mu}$, and consequently $|I| \leq  \sum_{i=0}^{d-1} n^{i+1+(d-i)\mu} \leq d n^{d+\mu}$.
\end{proof}

From \Cref{lem:DichotomyHyperplaneDependent} we deduce the following corollary.

\begin{corollary}
\label{cor:Dichotomy2}
    Let $0 <\mu < 1$ and $\points$ be a multiset of $n$ points in $\mathbb{R}^d$. Then there exists a subspace $\Pi_{d'}$ of dimension $0 \leq d' \leq d$ such that $|\points \cap \Pi_{d'}| \geq n^{\mu^{d-d'}}$ and the number of $d'$-dependent families of size $d'+1$ in $\points \cap \Pi_{d'}$ is at most $d|\points \cap \Pi_{d'}|^{d'+\mu}$.
\end{corollary}

\begin{proof}
We prove this corollary by induction on $d$. For $d=0$, the corollary is trivially true. Suppose now $d \geq 1$. We apply \Cref{lem:DichotomyHyperplaneDependent} and distinguish whether the first or second outcome of it is realised.

In the first case, there exists a subspace $\Pi_{d'}$ of dimension $0 \leq d' \leq d-1$ such that $|V \cap \Pi_{d'}| \geq n^{\mu}$. We now apply the induction hypothesis on $V'=V \cap \Pi_{d'}$, giving a subspace $\Pi_{d''}$ of dimension $d''$ such that $0 \leq d'' \leq d' \leq d-1$ and $|V \cap \Pi_{d''}| = |V' \cap \Pi_{d''}| \geq |V'|^{\mu^{d'-d''}} \geq n^{\mu^{d-d''}}$ and the number of $d''$-dependent families of $V \cap \Pi_{d''}=V' \cap \Pi_{d''}$ is at most $d|V \cap \Pi_{d''}|^{d''+\mu}$, as wanted.

In the second case, we immediately get the desired outcome for $d'=d$ and $\Pi_d=\mathbb{R}^d$.

Therefore, we are finished by induction.
\end{proof}

By the corollary above, for any set of $n$ points $\points$ in $\mathbb{R}^d$, there is a subspace $\Pi_{d'}$ that contains many of these points and such that there are not many $d'$-dependent families in $\points \cap \Pi_{d'}$.
The aim of the next lemma is to give sufficient conditions to reconstruct $\dist{v,\Pi_{d'}}$ and the projections of $v$ onto $\Pi_{d'}$ for almost all of $v \in V$ \whp, when the pairwise distances are revealed as $G(n,p)$ for sufficiently large $p$.

 \begin{lemma}
 \label{lem:RecoveringDistanceToHyperplane}
 Let $d \geq 1$ be an integer and $\newr$, $\delta$, $\lambda>0$ be such that $\newr-(d+1)\delta \geq \lambda$. Let $\points$ be a multiset of \(n\) points in \(\mathbb{R}^d\). Let $A \subset \points$ and $p \gg p_*(n,d)$. Suppose that 
     \begin{enumerate}
        \item there exists a $d'$-dimensional subspace $\Pi_{d'}$ in $\mathbb{R}^d$ such that $A \subset \Pi_{d'}$ and at least $\newr\binom{|A|}{d'+1}$ of sets of \(d'+1\) points of $A$ are $d'$-independent,
        \item $p|A| \gg 1$, and
        \item there exists a set $A' \subseteq A$ such that $|A'|\geq (1-\delta/2)|A|$ and for every $u,v \in A'$, we know the distance $\dist{u,v}$.
     \end{enumerate}
     If the distances are revealed according to a random graph $G(n,p)$, then {\whp}\ there exists a subset $V' \subseteq \points$ of size $|V'| \geq (1-\delta)|\points|$ such that for every $v' \in V'$, we can reconstruct $\dist{v', \Pi_{d'}}$ and the projection of \(v'\) onto \(\Pi_{d'}\) relative to the embedding of \(A'\).
 \end{lemma}

 \begin{proof}
 Let $\mathcal{I}$ be the set of $d'$-independent sets in $A'$. Note that $|\mathcal{I}|\geq \newr\binom{|A|}{d'+1}-\delta/2|A|\binom{|A|}{d'}$, and that $\binom{|A|}{d'+1} = \binom{|A|}{d'}\frac{|A|-d'-1}{d'+1} \geq \binom{|A|}{d'}\frac{|A|}{2(d+1)}$ when $n$ is large enough. Therefore, we have $|\mathcal{I}| \geq \lambda \binom{|A|}{d'+1}$ for $n$ large enough.
 
 Let $B=\points \setminus A$. Let $u \in B$, and for each $I \in \mathcal{I}$ we define the random variable $I_u$ to be the indicator of $u$ having an edge to each element of $I$. We also define the random variable $V_u=\sum_{I \in \mathcal{I}} I_u$. Since $|A|p \to\infty$, we have

 \[\mathbb{E}V_u\geq p^{d'+1}{|A|\choose{d'+1}}\lambda =  \Omega_{d, \lambda}\left(\left(|A|p\right)^{d'+1}\right), \mbox{ and }\]

 \[\begin{split}\Var V_u &{}=\sum_{I,I' \in \mathcal{I}}\left(p^{\left|I\cup I'\right|}-p^{2\left(d'+1\right)}\right)\\&{}\leq\sum_{I,I'\in A^{(d'+1)}}\left(p^{\left|I\cup I' \right|}-p^{2\left(d'+1\right)}\right)\\& =\sum_{i=0}^{d'+1}{|A|\choose{d'+1+i}}{{d'+1+i}\choose{i}}{{d'+1}\choose{i}}(p^{d'+1+i}-p^{2(d'+1)})\\&
 <\sum_{i=0}^{d'}{|A|\choose{d'+1+i}}{{d'+1+i}\choose{i}}{{d'+1}\choose{i}}p^{d'+1+i}\\&{} {}=O_{d}\left(\left(|A|p\right)^{2d'+1}\right).\end{split}\]

Let $B'=\{ u \in B: V_u > 1 \}$. By Chebyshev's inequality, we have \(\mathbb{P}\left(V_u=0\right)=O_{d,\lambda}\left(\left(|A|p\right)^{-1}\right) \to 0\), and therefore $\mathbb{E}|B'|=(1-o(1))|B|$. Hence, {\whp}\ $|B'| \geq (1-\delta/2)|B|$. Now by \Cref{lem:GeometryFact}, for every $b' \in B'$, there is an \(I\in\mathcal{I}\) such that we can reconstruct \(\{b'\}\cup I\). Since \(I\) is \(d'\)-independent, we can extend this embedding uniquely to reconstruct \(\{b'\}\cup A'\). Thus we can reconstruct $\dist{b', \Pi_{d'}}$ and the projection of \(b'\) onto \(\Pi_{d'}\) relative to the embedding of \(A'\). Letting $V'=B' \cup A'$, we reach the desired conclusion.
 \end{proof}

 We now present the following technical lemma, which allows us to reduce the problem to a lower dimension.

 \begin{lemma}
 \label{lem:InductionTechnicalLemma}
 Let $d \geq 1$, $0<\newdelta<\frac{1}{2(d+1)}$ be fixed and $V=\{v_1, \dots, v_n\}$ be a multiset of \(n\) points in \(\mathbb{R}^d\). Suppose the graph \(G=(V,\mathcal{P})\) of revealed pairwise distances is distributed as \(G(n,p)\) and that \(p\gg p_*(n,d) \). Then, {\whp}, we can explicitly construct some constants $b_{i,j}$ such that there exists a set $\newI \subseteq [n]$ such that $|\newI| \geq (1-\newdelta)n$ and $\tilde{V}=\{\tilde{v}_1, \dots, \tilde{v}_n \}$ a multiset of \(\mathbb{R}^{d'}\) for some $d' < d$ such that for every $i,j \in \newI$, we have $\dist{v_i,v_j}^2=b_{i,j}+\dist{\tilde{v}_i,\tilde{v}_j}^2$.
 \end{lemma}

Note that while we obtain the values of the \(b_{i,j}\), we do not claim to know the configuration of $\tilde{V}$. However, given \(b_{i,j}\), we can compute \(\dist{\tilde{v}_i,\tilde{v}_j}\) from \(\dist{v_i,v_j}\) and vice versa. Before proving \Cref{lem:InductionTechnicalLemma}, we show that a simple iteration of \Cref{lem:InductionTechnicalLemma} enables us to deduce \Cref{thm:main}.

 \begin{proof}[Proof of \Cref{thm:main} assuming \Cref{lem:InductionTechnicalLemma}]
 We let $q(n,d)=p_{*}(n,d)$ and $p \gg p_{*}(n,d)$. Note that $G$ sampled according to $G(n,p)$ can also be sampled as $\cup_{i=0}^{d-1} G_i$ where the $G_i$ are independent from each other and each $G_i$ is sampled according to $G(n,p')$ where $1-p=(1-p')^d$. By Bernoulli's inequality, $p' \geq p/d$, and therefore $p' \gg p_{*}(n,d)$.

     Let $\kappa>0$, and let $\newdelta=\min(\frac{1}{2(d+1)},\kappa/d)$. We start by applying \Cref{lem:InductionTechnicalLemma} to $V^{(0)}=V$ with the random graph $G_0$, and {\whp}\ get explicit constants $b^{(0)}_{i,j}$, a set $\newI^{(0)} \subseteq [n]$ such that $|\newI^{(0)}| \geq (1-\newdelta)n$, an integer $d^{(1)}<d$ and a set $V^{(1)} =\{v^{(1)}_1, \dots, v^{(1)}_n\}$ lying in $\mathbb{R}^{d^{(1)}}$ such that for every $i,j \in \newI^{(0)}$, we have $\dist{v^{(0)}_i,v^{(0)}_j}^2=b^{(0)}_{i,j}+\dist{v^{(1)}_i,v^{(1)}_j}^2$. If $d^{(1)} \neq 0$, we apply \Cref{lem:InductionTechnicalLemma} again to $V^{(1)}$ with the random graph $G_1$, and get explicit constants $b^{(1)}_{i,j}$, a set $\newI^{(1)} \subseteq [n]$ such that $|\newI^{(1)}| \geq (1-\newdelta)n$, an integer $d^{(2)}<d^{(1)}$ and a set $V^{(2)}=\{v^{(2)}_1, \dots, v^{(2)}_n\}$ lying in $\mathbb{R}^{d^{(2)}}$ such that for every $i,j \in \newI^{(2)}$, we have $\dist{v^{(1)}_i,v^{(1)}_j}^2=b^{(1)}_{i,j}+\dist{v^{(2)}_i,v^{(2)}_j}^2$.
     
     We continue in the same fashion until we have $d^{(k)}=0$ for some integer $k$. As $d^{(i+1)} \leq d^{(i)} -1$ for all $i \leq k-1$, we have $k \leq d$. Let $\newI=\newI^{(0)} \cap \dots \newI^{(k-1)}$. We have $|\newI|=|\newI^{(0)} \cap \dots \newI^{(k-1)}| \geq (1-d\newdelta)n \geq (1-\kappa)n$, and for every $i,j \in \newI$, it follows that $\dist{v_i,v_j}^2=\dist{v^{(0)}_i,v^{(0)}_j}^2=b^{(0)}_{i,j}+ \dist{v^{(1)}_i,v^{(1)}_j}^2= \dots =b^{(0)}_{i,j}+ \dots + b^{(k-1)}_{i,j}$. Therefore we can reconstruct all the pairwise distances in $\newI$ {\whp} This finishes the proof.
 \end{proof}

All there is left is to prove \Cref{lem:InductionTechnicalLemma}.

\begin{proof}[Proof of \Cref{lem:InductionTechnicalLemma}]
 Let $p \gg p_{*}(n,d)$. Note that $G$ sampled according to $G(n,p)$ can also be sampled as $G_1 \cup G_2 \cup G_3$ where $G_1$, $G_2$ and $G_3$ are independent from each other and $G_1$, $G_2$ and $G_3$ are each sampled according to $G(n,p')$ where $1-p=(1-p')^3$. By Bernoulli's inequality, $p' \geq p/3$, and therefore $p' \gg p_{*}(n,d)$.
By \Cref{cor:Dichotomy2} applied to $\points$ and $(1/\eta(d))^{1/d}<\mu<1$ there exists a subspace $\Pi_{d'}$ of dimension $0 \leq d' \leq d-1$ such that $|V \cap \Pi_{d'}| \geq n^{\mu^{d-d'}}$ and the number of $d'$-dependent families of $\points \cap \Pi_{d'}$ is at most $d|\points \cap \Pi_{d'}|^{d'+\mu}$. 

If $d' \geq 1$, let $A=V \cap \Pi_{d'}$ and let $\mu'$ be fixed such that $\mu < \mu' < 1$. The number of $d'$-dependent families of $A$ is at most $d|A|^{\mu+d'} \leq |A|^{\mu'+d'}$ for $n$ large enough. Therefore by \Cref{thm:PollutedBootstrapPercolation} with the pairwise distances revealed as $G_1$, {\whp}\ we can reconstruct the pairwise distances within a set $A' \subseteq A$ of size $|A'| \geq (1-\newdelta/2)|A|$. By \Cref{lem:RecoveringDistanceToHyperplane} with the pairwise distances revealed as $G_2$, since $p'|A| \geq p' n^{\mu^d} \gg 1$, {\whp}\ there is a set $W \subseteq \points$ such that $|W| \geq (1-\newdelta)|\points|$ and for every element of $W$, we can reconstruct the distance to $\Pi_{d'}$ and the projection onto \(\Pi_{d'}\) relative to the embedding of \(A'\).
For every $v_i \in W$, let $v^{\parallel}_i$ be the projection of $v_i$ on $\Pi_{d'}$ and $\tilde{v}_i=v_i-v^{\parallel}_i$. For every $v_i, v_j \in W$, let $b_{i,j}=\dist{v^{\parallel}_i, v^{\parallel}_j}^2$, and note that we have knowledge of the value of $b_{i,j}$. The $\tilde{v}_i$ lie inside of a space isomorphic to $\mathbb{R}^{d-d'}$, and for every $v_i, v_j \in W$, we have $\dist{v_i,v_j}^2=b_{i,j}+\dist{\tilde{v}_i,\tilde{v}_j}^2$. This finishes the first case.

If $d'=0$, we have a point of multiplicity $m \geq n^{\mu^{d}}$ in $\points$, and let $M$ be the multiset corresponding to it. If $m \geq (1-\newdelta)n$, then as $p' \geq p_{*}(n,d)/3 \geq 4\frac{\log n}{n} \geq 2\frac{\log m}{m}$ for $n$ large enough, by \Cref{thm:ConnectivityErdosRenyi} we have that {\whp}\ $G_1$ restricted to the set $M$ is connected. Therefore we can reconstruct the set $M$, as wanted. If $m \leq (1-\newdelta)n$, then $Y=\points \setminus M$ has size $|Y|\geq \newdelta n$. We apply \Cref{cor:Dichotomy2} again, but this time to $Y$. Therefore there exists a subspace $\Pi_{d''}$ of dimension $0 \leq d'' \leq d-1$ such that $|Y \cap \Pi_{d''}| \geq \newdelta^{\mu^{d-d''}}n^{\mu^{d-d''}}$ and the number of $d''$-dependent families of $Y \cap \Pi_{d''}$ is at most $d |Y \cap \Pi_{d''}|^{d''+\mu}$. If $d'' \geq 1$, then one concludes the same way as in the case $d' \geq 1$.

If $d''=0$, then there exists a point of multiplicity $m' \geq |Y|^{\mu^{d}} \geq \newdelta^{\mu^{d}}n^{\mu^{d}}$ in $Y$, and we let $M'$ be the multiset corresponding to it. Note that for $n$ large enough we have $p' \geq p_{*}(n,d)/3 \geq 2\frac{\log n}{\newdelta n} \geq 2\frac{\log m'}{m'}$, and therefore by \Cref{thm:ConnectivityErdosRenyi} with the pairwise distances revealed as $G_2$ {\whp}\ we can reconstruct the set $M'$. Let $\Pi_1$ be the line in $\mathbb{R}^d$ containing both $M$ and $M'$ and let $M''=M \cup M'$. Clearly, \whp\ the distance between $M$ and $M'$ is among the revealed distances.
Note that $p|M| \gg 1$ and $p|M'| \gg 1$, so we may apply \Cref{lem:RecoveringDistanceToHyperplane} twice, once for $A=M$, and once for $A=M'$, with $d'=0$, $\newr=1$, $\delta=\newdelta/2$, $\lambda=1/2$. We obtain sets $X, X' \subset \points$ such that $|X| \geq (1-\newdelta/2)n$ and $|X'|\geq (1-\newdelta/2)n$, and for each $x \in X$, we have reconstructed $\dist{x,M}$ and for each $x' \in X'$, we have reconstructed $\dist{x',M'}$. Letting $Z=X \cap X'$, we have $|Z| \geq (1-\newdelta)n$, and for every $z \in Z$, we have reconstructed $\dist{z,\Pi_1}$ and the projection $z'$ of $z$ onto $\Pi_1$ relative to the embedding of $M''$. From there, we can conclude the same way as in the first outcome above. This finishes the proof.
\end{proof}

\section{Concluding remarks}
\label{sec:Conclusion}

For any set of $n$ points $\points$ in $\mathbb{R}^d$, we have obtained in \Cref{thm:main} a non-trivial bound on $p$ under which, {\whp}, we can reconstruct the pairwise distances within a subset of $\points$ of size $n-o(n)$. However, we do not believe that our bound is tight. For instance, for $d=1$, much stronger results have been obtained in \Cref{thm:GIMPSLinear} by Gir\~ao, Illingworth, Michel, Powierski, and Scott, suggesting that the criticial probability could be $c_d/n$ for some constant $c_d$. Additionally, we have sought to reconstruct \(n-o(n)\) points, rather than \(\Omega(n)\) as posed by them in \Cref{qu:GIMPS}---we do not know whether the thresholds of these two problems are different.

We remark that Bartha and Kolesnik \cite{bartha_weakly_2023} recently improved upon Balogh, Bollob\'as, and Morris's result \Cref{thm:Balogh} and determined $p_c(n,K_{d+3})$ up to a multiplicative constant, i.e., \(p_c(n,K_{d+3}) = \Theta\left(n^{-1/\eta(d)}\right)\). It is possible that adapting their method to the reconstructibility problem could improve the result of \Cref{thm:main} up to a polylogarithmic factor, but such attempts were not made in this paper, since, as mentioned earlier, we do not believe that the range of $p$ in \Cref{thm:main} is close to optimal. We would also like to remark that, instead of \Cref{thm:PollutedBootstrapPercolation}, one could have applied a second moment method to find the number of embeddings of \(K_{d+3}-e\) rooted on each edge. This corresponds to taking \(r=1\) in this analysis, and would give a slightly larger exponent of \(n\) in \(q(n,d)=p_{*}(n,d)\).

\section{Acknowledgements}
The research was mostly conducted while the first, the fourth and the fifth authors were students taking part in the University of Cambridge Summer Research in Mathematics (SRIM) programme, and were supervised by Julian Sahasrabudhe. The authors would like to thank Julian Sahasrabudhe for the many hours of helpful discussions, valuable comments on the manuscript and for suggesting the problem.

The first and fourth authors were supported by the Trinity College Summer Studentship Scheme fund. The fifth author was supported by the Churchill College Summer Opportunities Bursary and the CMS Bursary. The second and third author are supported by EPSRC (Engineering and Physical Sciences Research Council): the second author's reference is EP/V52024X/1, the third author's reference is EP/T517847/1. The second author is also supported by the Department of Pure Mathematics and Mathematical Statistics of the University of Cambridge, and the third author is also supported by the Cambridge Commonwealth, European and International Trust.

\bibliographystyle{amsplain}  
\renewcommand{\bibname}{Bibliography}
\bibliography{brsbibtex}

\providecommand{\bysame}{\leavevmode\hbox to3em{\hrulefill}\thinspace}
\providecommand{\MR}{\relax\ifhmode\unskip\space\fi MR }
\providecommand{\MRhref}[2]{%
  \href{http://www.ams.org/mathscinet-getitem?mr=#1}{#2}
}
\providecommand{\href}[2]{#2}
\begin{thebibliography}{10}

\bibitem{balogh_graph_2012}
J{\'o}zsef Balogh, B{\'e}la Bollob{\'a}s, and Robert Morris, \emph{Graph bootstrap percolation}, Random Structures \& Algorithms \textbf{41} (2012), no.~4, 413--440.

\bibitem{bartha_weakly_2023}
Zsolt Bartha and Brett Kolesnik, \emph{Weakly saturated random graphs}, April 2023, arXiv:2007.14716 [math].

\bibitem{benjamini_determining_2022}
Itai Benjamini and Elad Tzalik, \emph{Determining a {Points} {Configuration} on the {Line} from a {Subset} of the {Pairwise} {Distances}}, October 2022, arXiv:2208.13855 [math].

\bibitem{bollobas1968weakly}
B{\'e}la Bollob{\'a}s, \emph{Weakly k-saturated graphs}, Beitr{\"a}ge zur Graphentheorie (Kolloquium, Manebach, 1967), vol.~25, 1968, p.~31.

\bibitem{bollobas1998random}
\bysame, \emph{Random graphs}, Springer, 1998.

\bibitem{bollobas_threshold_1987}
B{\'e}la Bollob{\'a}s and Andrew~G. Thomason, \emph{Threshold functions}, Combinatorica \textbf{7} (1987), no.~1, 35--38.

\bibitem{connelly2005generic}
Robert Connelly, \emph{Generic global rigidity}, Discrete \& Computational Geometry \textbf{33} (2005), 549--563.

\bibitem{supersaturation}
Paul Erd\H{o}s and Mikl\'{o}s Simonovits, \emph{Supersaturated graphs and hypergraphs}, Combinatorica \textbf{3} (1983), 181--192.

\bibitem{erdHos1960evolution}
Paul Erd{\H{o}}s and Alfr{\'e}d R{\'e}nyi, \emph{On the evolution of random graphs}, Publ. Math. Inst. Hung. Acad. Sci \textbf{5} (1960), no.~1, 17--60.

\bibitem{eren_rigidity_2004}
Tolga Eren, David Goldenberg, Walter Whiteley, Yang Yang, A.~Morse, Brian Anderson, and Peter Belhumeur, \emph{Rigidity, {Computation}, and {Randomization} in {Network} {Localization}.}, vol.~4, January 2004.

\bibitem{girao_reconstructing_2023}
António Girão, Freddie Illingworth, Lukas Michel, Emil Powierski, and Alex Scott, \emph{Reconstructing a point set from a random subset of its pairwise distances}, January 2023, arXiv:2301.11019 [math].

\bibitem{gortler_characterizing_2012}
Steven~J. Gortler, Alexander~D. Healy, and Dylan~P. Thurston, \emph{Characterizing generic global rigidity}, American Journal of Mathematics \textbf{132} (2010), no.~4, 897--939.

\bibitem{gravner1997bootstrap}
Janko Gravner and Elaine McDonald, \emph{Bootstrap percolation in a polluted environment}, Journal of Statistical Physics \textbf{87} (1997), 915--927.

\bibitem{hendrickson_molecule_1990}
Bruce~Alan Hendrickson, \emph{The molecule problem: Determining conformation from pairwise distances}, Cornell University, 1991.

\bibitem{horn2013matrix}
Roger~A. Horn and Charles~R. Johnson, \emph{Matrix analysis}, Cambridge University Press, 2013.

\bibitem{jordan2017global}
Tibor Jord{\'a}n and Walter Whiteley, \emph{Global rigidity}, Handbook of Discrete and Computational Geometry, Chapman and Hall/CRC, 2017, pp.~1661--1694.

\bibitem{lemke_reconstructing_2003}
Paul Lemke, Steven~S. Skiena, and Warren~D. Smith, \emph{Reconstructing sets from interpoint distances}, Discrete and Computational Geometry: The Goodman-Pollack Festschrift, Springer, 2003, pp.~597--631.

\bibitem{lew_sharp_2022}
Alan Lew, Eran Nevo, Yuval Peled, and Orit~E. Raz, \emph{Sharp threshold for rigidity of random graphs}, Bulletin of the London Mathematical Society \textbf{55} (2023), no.~1, 490--501.

\bibitem{neumann1966theory}
John~von Neumann, \emph{Theory of self-reproducing automata}, Edited by Arthur W. Burks (1966).

\bibitem{ulam1952random}
Stanislaw Ulam, \emph{Random processes and transformations}, Proceedings of the international congress on mathematics, vol.~2, Citeseer, 1952, pp.~264--275.

\end{thebibliography}

\appendix
\section{Appendix}
\label{appendix}

We now prove \Cref{lem:GeometryFact} and its \Cref{cor:GeometryFact}, repeated here for ease of reading.
\renewcommand\thetheorem{2.2}

\begin{lemma}
        Given all pairwise distances within some \(S=\{v_1,\dots,v_r\}\subset\mathbb{R}^d\), we can reconstruct the configuration of \(S\) up to isometry.
	\end{lemma}

 \begin{proof}
        Write \(u_i=v_i-v_r\) for \(1\leq i< r\). Recall that the Gram matrix of the vectors \(u_1,\dots,u_{r-1}\) is the matrix given by \(G_{ij}=\langle u_i,u_j\rangle\). For \(i=j\), this is \(|v_i-v_r|^2\). For \(i\neq j\), we have \(|v_i-v_j|^2=|u_i-u_j|^2=G_{ii}+G_{jj}-2G_{ij}\), so \(G_{ij}=\frac12\left(|v_i-v_r|^2+|v_j-v_r|^2-|v_i-v_j|^2\right)\). Thus we can compute the Gram matrix of \(u_1,\dots,u_{r-1}\) from the distances within \(S\).

        The Gram matrix of a set of points determines their positions up to orthogonal transformations (see for instance Theorem 7.3.11, \cite{horn2013matrix}). So the pairwise distances within \(S\) determine \(u_i\) up to orthogonal transformations, and thus determine \(S\) up to isometry.
	\end{proof}
 
\renewcommand\thetheorem{2.3}
	
	\begin{corollary}
		Suppose a multiset \(R\) of \(d+3\) points of \(\mathbb{R}^m\) \((m\geq d\)) lies in a \(d\)-dimensional subspace. Then, given all pairwise distances within \(R\) except some pair \(uv\), it is possible to identify whether \(\newT=R\setminus\{u,v\}\) is \(d\)-independent. If $Q$ is $d$-independent, the distance \(uv\) can also be determined.
	\end{corollary}

\begin{proof}[Proof of \Cref{cor:GeometryFact}]
    From these distances, \Cref{lem:GeometryFact} allows us to reconstruct \(\newT\) up to isometry. This allows    to identify whether \(\newT\) is \(d\)-independent by checking the reconstruction of \(\newT\). Indeed, \(\newT\) is \(d\)-independent if and only if the \(u_i\) as in the proof of \Cref{lem:GeometryFact} are linearly independent.
    This holds if and only if \(G\) is positive definite (see for instance Theorem 7.2.10, \cite{horn2013matrix}), which can be computed and checked.

    Suppose \(\newT\) is \(d\)-independent. Note that by \Cref{lem:GeometryFact} it is also possible to reconstruct each of \(\newT\cup\{u\}\) and \(\newT\cup\{v\}\), up to isometry. But \(u,v\) lie in the \(d\)-dimensional subspace spanned by \(\newT\), so any isometry fixing \(\newT\) also fixes \(u,v\). So given an embedding of \(\newT\), this can be uniquely extended to embed \(u,v\). Thus it is possible to determine the distance \(uv\).
\end{proof}

\end{document}